\documentclass[11pt, reqno]{amsart}
\usepackage[utf8]{inputenc}	
\usepackage{amssymb,amscd}
\usepackage{multirow,booktabs}
\usepackage[table]{xcolor}
\usepackage{fullpage}
\usepackage{lastpage}
\usepackage{fancyhdr}
\usepackage{mathrsfs}
\usepackage{wrapfig}
\usepackage{graphicx}
\usepackage{setspace}
\usepackage{calc}
\usepackage{bm}
\usepackage{tikz}
\usetikzlibrary{arrows.meta, positioning}
\usepackage{cite}
\usepackage{enumerate}
\usepackage{subfigure}
\usepackage{multicol}
\usepackage[colorlinks,
            linkcolor=blue,
            citecolor=red,
            ]{hyperref}
\usepackage{cancel}
\usepackage[all]{xy}
\usepackage{geometry}
\usepackage{empheq}
\usepackage{framed}
\usepackage{xcolor}
\hypersetup{pdftitle={Transition probabilities of step-reinforced random walks},
 pdfauthor={Yuval Peres and Shuo Qin}}

\theoremstyle{plain}
\newtheorem{theorem}{Theorem}[section]
\newtheorem{lemma}[theorem]{Lemma}
\newtheorem{proposition}[theorem]{Proposition}
\newtheorem{corollary}[theorem]{Corollary}

\theoremstyle{remark}
\newtheorem{definition}{Definition}
\newtheorem{example}{Example}[section]
\newtheorem{remark}{Remark}[section]

\usepackage{indentfirst}
\begin{document}

\newcommand{\QQ}{\mathbb{Q}}
\newcommand{\RR}{\mathbb{R}}
\newcommand{\ZZ}{\mathbb{Z}}
\newcommand{\NN}{\mathbb{N}}
\newcommand{\CC}{\mathbb{C}}
\newcommand{\EE}{\mathbb{E}}
\newcommand{\Var}{\operatorname{Var}}
\newcommand{\CalC}{\mathcal{C}}
\newcommand{\PP}{\mathbb{P}}
\newcommand{\Rd}{\mathbb{R}^d}
\newcommand{\Rn}{\mathbb{R}^n}
\newcommand{\FF}{\mathcal{F}}
\newcommand{\GG}{\mathcal{G}}
\newcommand{\F}{\mathscr{F}}

\author[Yuval Peres]{Yuval Peres}
\address[Yuval Peres]{Beijing Institute of Mathematical Sciences and Applications}
\email{yperes@gmail.com}
  
\author[Shuo Qin]{Shuo Qin}
\address[Shuo Qin]{Beijing Institute of Mathematical Sciences and Applications, and Yau Mathematical Sciences Center, Tsinghua University}
\email{qinshuo@bimsa.cn}

\title{Transition probabilities of step-reinforced random walks}
\date{}

\begin{abstract}
The step-reinforced random walk (SRRW) either repeats a uniformly chosen past step or takes a fresh independent step. We consider a generalized SRRW on groups, where the selected past step is transformed by a random map. The transformations are independent of the fresh steps but may be dependent on one another.

For every reinforcement parameter $\alpha<1$, we obtain upper bounds on transition probabilities in terms of the geometry of the group. On $\mathbb{R}^d$, a non-singular step distribution yields bounds of order $n^{-d/2}$ for balls of any fixed radius, uniformly over their centers and without any moment assumption, and hence transience for $d\geq 3$. On finitely generated groups, we give bounds in terms of the isoperimetric profile. We also prove exponential decay for the elephant random walk on every nonamenable Cayley graph with a finite symmetric generating set and every memory parameter $p<1$. In particular, this answers a question of Mukherjee for Cayley trees.
\end{abstract}

\maketitle

\section{Introduction}

\subsection{Definitions and related models}
\label{secdef}

In recent years, the step-reinforced random walk has attracted considerable attention. At each time step, the walk either replicates a uniformly chosen step from its past or takes a fresh step independent of the history. In this paper, we consider the following generalization, where the selected step may be transformed rather than simply repeated. Throughout, we assume that $(G,\cdot)$ is either the additive group $(\mathbb{R}^d,+)$ equipped with the Borel $\sigma$-algebra or a discrete group, and assume that $\mu$ is a probability measure on $G$.

\begin{definition}[A generalized SRRW on a group]
\label{defSRRWG}
Let $(g_n)_{n\geq 1}$ be i.i.d. random variables with law $\mu$, let $(\xi_n)_{n\geq 2}$ be i.i.d. Bernoulli random variables with parameter $\alpha\in[0,1]$, and let $(u_n)_{n\geq 2}$ be independent, with $u_n$ uniform on $\{1,\ldots,n-1\}$. Let $(T_n)_{n\geq 2}$ be random measurable transformations of $G$, where $(\omega,g)\mapsto T_n(\omega,g)$ is jointly measurable. We assume that the four sequences
\[
(g_n)_{n\geq 1},\qquad (\xi_n)_{n\geq 2},\qquad
(u_n)_{n\geq 2},\qquad (T_n)_{n\geq 2}
\]
are mutually independent. No independence among the transformations $T_n$ is required.

Set $S_0=e_G$, $X_1=g_1$, and, for $n\geq 2$, define
\[
X_n=\begin{cases}
g_n,&\xi_n=0,\\
T_n(X_{u_n}),&\xi_n=1.
\end{cases}
\qquad S_n=X_1\cdot X_2\cdots X_n\quad(n\geq 1).
\]
We call $S=(S_n)_{n\geq 0}$ a generalized step-reinforced random walk on $G$ with reinforcement parameter $\alpha$, step distribution $\mu$, and transformations $(T_n)_{n\geq 2}$.
\end{definition}

\begin{remark}
\label{remforestdependent}
The results below for generalized SRRWs also hold if the transformations depend on the past of $(\xi_n,u_n)_{n\geq 2}$, but not on the fresh steps. More precisely, take a random element $\Theta$ independent of the three mutually independent sequences $(g_n)$, $(\xi_n)$ and $(u_n)$, and let
\[
T_n(g)=\mathscr T_n\bigl(\Theta,(\xi_j,u_j)_{2\leq j<n},g\bigr)
\]
for measurable maps $\mathscr T_n$. The sequence $(g_n)$ is then independent of the entire forest and all transformations. This independence is used in the forest representation below; the dependence on past forest variables also ensures that the walk is adapted to the filtration used for the conditional estimates. 
\end{remark}

The generalized SRRW includes many existing models.
When $T_n = \operatorname{Id}$ for all $n\geq 2$, the walk $S$ is the usual SRRW on groups (see \cite[Definition 1.1]{peres2026mixing}). If $G=(\RR^d, +)$ and $(T_n)_{n\geq 2}$ are linear transformations on $\RR^d$, i.e., $T_n(X):=A_nX$ where $(A_n)_{n\geq2}$ is a sequence of $d\times d$ random matrices independent of the fresh steps and reinforcement variables, then the following models are special cases of such generalized SRRWs:
\begin{itemize}
    \item the random walk with counterbalanced steps introduced by Bertoin \cite{MR4756572}, where $A_n \equiv -I_d$;
    \item the unbalanced step-reinforced random walk introduced by Aguech, Ben Hariz, El Machkouri, and Faouzi \cite{AGUECH2026105026}, where $(A_n)_{n\geq 2}$ are i.i.d. and take values in $\{I_d,-I_d\}$;
    \item the random walk with echoed steps introduced by Portillo del Valle \cite{del2025random}, where $d=1$, $A_n=a_n$, and $(a_n)_{n\geq 2}$ are i.i.d.\ real random variables with the echo law;
   \item the elephant random walk (ERW) introduced by Schütz and Trimper \cite{schutz2004elephants} and its multidimensional version by Bercu and Laulin \cite{MR3962977}, where $A_n$ is given by \cite[Equation (2.1)]{MR3962977}. 
\end{itemize}

Mukherjee \cite{mukherjee2025elephant} recently extended the ERW to finitely generated groups: Suppose $G$ is a group with a finite symmetric generating set $\Gamma$ containing at least two elements. Let $G_{\Gamma}$ denote the Cayley graph of $G$ with respect to $\Gamma$. The first step of the ERW on $G_{\Gamma}$ is sampled uniformly at random from $\Gamma$. At each time step $n\geq 2$, the elephant chooses a step from the past uniformly at random, say $X_{u_n}$, and then, with probability $p$ (which is called the memory parameter), repeats this step; otherwise, the next step is sampled uniformly from $\Gamma \backslash \{X_{u_n}\}$. This extension is also a special case of the generalized SRRW, see Lemma \ref{ERWgSRRW}.

Note that when $\alpha=1$, for $n\geq 2$, one has $S_n=X_1 \cdot T_2(X_{u_2}) \cdot T_3(X_{u_3})\cdots T_n(X_{u_n})$. Hence, the asymptotic behavior of $S$ strongly depends on the choice of $(T_n)_{n\geq 2}$. In this paper, \emph{we focus on the case $\alpha<1$ and aim to obtain upper bounds on the transition probabilities of $S$ on infinite groups uniformly over the transformations allowed in Definition \ref{defSRRWG}} (and in fact, the statements of our main results will often omit the transformations $(T_n)_{n\geq 2}$).

\subsection{Main results}

For a generalized SRRW $S$ on $\RR^d$, we say that $S$ is transient if $\|S_n\|\to\infty$ almost surely, where $\|\cdot\|$ is the Euclidean norm. We call a probability measure $\mu$ on $\RR^d$ non-singular if it is not supported on any proper affine hyperplane, equivalently, \[ \operatorname{span}\bigl(\operatorname{supp}\mu-\operatorname{supp}\mu\bigr) =\mathbb R^d. \] This condition is stronger than $ \operatorname{span}(\operatorname{supp}\mu)=\mathbb R^d$. The two conditions are equivalent whenever $0\in\operatorname{aff}(\operatorname{supp}\mu)$, in particular when the support is symmetric about the origin, or when $\mu$ has a finite first moment and mean zero.

Proposition \ref{Strandabove3} gives transience in dimensions $d\geq 3$ for these models whenever they admit a representation with $\alpha<1$ and a non-singular step distribution. For the usual SRRW, transience in these dimensions was proved in \cite[Theorem 1]{MR5009036} under a finite $2+\delta$ moment assumption. Under the non-singularity assumption, the following result requires no moment condition.

\begin{proposition}
\label{Strandabove3}
Let $S$ be a generalized SRRW on $\RR^d$ with reinforcement parameter $\alpha\in[0,1)$ and a non-singular step distribution $\mu$. Then for every $r>0$ there is a constant $C=C(r,\mu,\alpha)>0$ such that
\[
\sup_{x\in\RR^d}\PP(S_n\in B(x,r))\leq Cn^{-d/2},\qquad n\geq 1,
\]
where $B(x,r)$ is the open Euclidean ball of radius $r$ centered at $x$. In particular, $S$ is transient if $d\geq 3$.
\end{proposition}

For a discrete group $G$, we say that $\mu$ is a class function if it is constant on conjugacy classes of $G$, i.e., 
$$ \mu(y^{-1}\cdot x\cdot y)=\mu(x), \quad  \forall x,y \in G. $$
Equivalently, $\mu( x\cdot y)=\mu(y\cdot x)$ for all $x,y \in G$. Note that if $G$ is abelian, then every probability measure on $G$ is a class function. If $\mu$ is a class function, Proposition \ref{propclassbd} shows that transition probabilities can be upper bounded by studying the non-reinforced chain $(\alpha=0)$.

\begin{proposition}
\label{propclassbd}
    Let $S$ be a generalized SRRW on a countably infinite group $G$ with parameter $\alpha \in [0,1)$ and step distribution $\mu$ being a class function. We denote its law by $\PP^{(\alpha)}$ to indicate the dependence on $\alpha$. Let $\varepsilon \in (0,1)$ and $m \in \NN_+$ be such that 
$$
\max_{x\in G} \PP^{(0)}(S_n=x) \leq \varepsilon, \quad \forall  n\geq m.
$$
Then
$$
\max_{x\in G} \PP^{(\alpha)}(S_n=x) \leq \varepsilon + 2e^{-(1-\alpha)n/18}, \quad \forall  n\geq \frac{8 m}{1-\alpha}.
$$
\end{proposition}
\begin{example}
Let $\mathcal{S}_3$ be the symmetric group on three elements, and let $G$ be the direct product of $\mathcal{S}_3$ and $(\ZZ,+)$. Assume $\mu$ is the uniform distribution on $\Gamma=\{((12),0),((13),0),((23),0)\} \cup \{(\text{Id},1),(\text{Id},-1)\}$. Let $S$ be as in Proposition \ref{propclassbd}. Since its projection to $\ZZ$ is a delayed version of a simple random walk on $\ZZ$, one has 
$$
\max_{x\in G} \PP^{(0)}(S_n=x)  \leq Cn^{-\frac{1}{2}}, \quad n\geq 1,
$$
for some universal constant $C$. Proposition \ref{propclassbd} then shows that $\max_{x\in G} \PP^{(\alpha)}(S_n=x) \leq C_1 n^{-\frac{1}{2}}$ for some positive constant $C_1=C_1(\alpha)$.
\end{example}

In the study of Markov chains, it is often assumed that the chain is lazy so that at each time step, the walker remains at his/her current position with positive probability. The following Theorem \ref{lazyGinf} shows that the transition probabilities for the reinforced version of lazy chains can be upper bounded via the so-called isoperimetric profile. For two subsets $A,B$ of a discrete group $G$, we write 
\begin{equation}
    \label{PmuABdef}
    P_{\mu}(A,B):=\sum_{x\in A,y\in B} P_{\mu}(x,y), \quad \text{where } P_{\mu}(x,y):=\mu(x^{-1}\cdot y).
\end{equation}
Following \cite{MR3726904}, for a non-empty finite subset $A \subset G$, we call $\Phi(A):=P_{\mu}(A,A^c)/|A|$ the bottleneck ratio of $A$. When $G$ is infinite, we define the isoperimetric profile $\Phi(r)$ for $r\geq 1$ by 
\begin{equation}
    \label{defphirGinf}
    \Phi(r):=\inf \{\Phi(A): 0<|A| \leq r\},\quad r \geq 1.
\end{equation}
We say that a generalized SRRW with step distribution $\mu$ is irreducible if $P_{\mu}$ given in (\ref{PmuABdef}) is irreducible.

\begin{theorem}[Lazy chains]
    \label{lazyGinf}
   Let $S=(S_n)_{n\in \NN}$ be an irreducible generalized SRRW on a countably infinite group $G$ with parameter $\alpha \in [0,1)$ and step distribution $\mu$ such that $\mu(e_G)\geq \mu_0$ for some $\mu_0\in (0,1/2]$. Then for any $\varepsilon \in (0,1)$, one has
$$
 \max_{x \in G} \PP(S_n=x) \leq \varepsilon \quad \text{if } n \geq \frac{C(\mu_0)}{1-\alpha} \int_{4}^{8 / \varepsilon}\frac{1 }{u \Phi^2(u)} d u, 
$$
  where $C(\mu_0)$ is a positive constant that depends only on $\mu_0$.
\end{theorem}

When $G$ is countable, let 
$$
\Gamma:=\{x\in G: \mu(x)>0\}
$$
be the support of $\mu$. If $\Gamma$ is finite and generates $G$, we use the undirected Cayley graph $G_{\Gamma}$ associated with $\Gamma\cup\Gamma^{-1}$ to describe the geometry of $G$. The generalized SRRW itself need not be nearest-neighbor unless $T_n(\Gamma)\subseteq\Gamma$ almost surely for every $n$. For a discrete group $G$, we say that $S$ is transient if it almost surely eventually leaves every finite subset of $G$. 

\begin{corollary}
     \label{corlazytran}
     Let $S$, $G$ and $\mu$ be as in Theorem \ref{lazyGinf} and assume that $\mu$ has finite support $\Gamma$ which generates $G$. \\
    (i).  If $G_{\Gamma}$ is of polynomial growth
$d\geq 1$ (e.g., $\ZZ^d$), then there exists a positive constant $C_1=C_1(G,\mu, \alpha)$ such that
    $$
\max_{x \in G} \PP(S_n=x) \leq C_1 n^{-\frac{d}{2}}, \quad \forall n\geq 1,
    $$
    In particular, if $G_{\Gamma}$ is of at least cubic growth, then $S$ is transient. \\
    (ii). If $G_{\Gamma}$ is of exponential growth (e.g., the lamplighter group over $\ZZ$), then  there exist positive constants $C_2=C_2(G,\mu, \alpha)$ and $C_3=C_3(G,\mu, \alpha)$ such that
 $$
\max_{x \in G} \PP(S_n=x) \leq C_2 e^{-C_3n^{1/3}}, \quad \forall n\geq 1,
    $$
    (iii). If $G_{\Gamma}$ is nonamenable, that is, 
    $$
  \inf\left\{\frac{|\{(x,y) \in E(G_{\Gamma}): x \in A, y \in A^c\}|}{|A|}: A\subset G,\ 0<|A|<\infty\right\} >0,
    $$
    then there exist positive constants $C_4=C_4(G,\mu, \alpha)$ and $C_5=C_5(G,\mu, \alpha)$ such that
    \begin{equation}
        \label{eqexpdecaytranpro}
        \max_{x \in G} \PP(S_n=x) \leq  C_4 e^{-C_5 n},  \quad \forall n\geq 1.
    \end{equation}
\end{corollary}

Corollary \ref{nonamentran} below shows that the exponential decay in (\ref{eqexpdecaytranpro}) also holds when the assumption $\mu(e_G)>0$ is replaced by the symmetry of $\Gamma$. This resolves an open question posed by Mukherjee (see \cite[Open problem 2.2]{mukherjee2025elephant}) that the $n$-step return probability of the ERW on a Cayley tree decays exponentially fast in $n$ (recall the definition of ERW on Cayley graphs from Section \ref{secdef}).

 \begin{corollary}
    \label{nonamentran}
   Suppose $G$ is a group generated by a finite symmetric set $\Gamma$ such that $G_{\Gamma}$ is nonamenable. \\
   (i) Let $S$ be an irreducible generalized SRRW on $G$ with parameter $\alpha \in [0,1)$ and step distribution $\mu$ whose support is $\Gamma$. Then, there exists a constant $\kappa=\kappa(G,\mu,\alpha) \in (0,1)$ such that 
    \begin{equation}
        \label{returnprobexpon}
         \sup_{x\in G} \PP(S_n=x) \leq \kappa^n, \quad \forall n\geq 1.
    \end{equation}
    In particular, $\liminf_n d(e_G,S_n)/n >0$ a.s. where $d(\cdot,\cdot)$ denotes the graph distance in $G_{\Gamma}$. \\
    (ii). Let $S$ be an ERW on $G_{\Gamma}$ with memory parameter $p\in[0,1)$. Then there exists a constant $\rho=\rho(G,\Gamma,p)\in(0,1)$ such that
    $$
\sup_{x\in G} \PP(S_n=x) \leq \rho^n, \quad \forall n\geq 1.
    $$
    Moreover, $\liminf_n d(e_G,S_n)/n>0$ almost surely.
\end{corollary}
\begin{remark}
\label{mukhexpdecay}
When $G_{\Gamma}$ is the infinite $d$-regular tree with $d\geq 3$, the exponential decay for $\PP(S_n=e_G)$ has been proved by Mukherjee for the case when $p<1/2$ and $(p,d)\neq (0,3)$, see \cite[Theorem 2.3]{mukherjee2025elephant}.
\end{remark}

Finally, we consider possibly the simplest non-trivial SRRW: Let $S$ be the usual SRRW on $G=(\ZZ_2, +)$ with reinforcement parameter $\alpha \in (0,1)$ and step distribution $\mu$ such that $\mu(1)= \mu(0)=1/2$. Observe that for any fixed even $n\geq 2$,  
$$
\lim_{\alpha \to 0+}\PP(S_n=0) = \frac{1}{2}, \quad \lim_{\alpha \to 1-}\PP(S_n=0) = \PP(n X_1 = 0)=1.
$$
Corollary \ref{twopointcor} below gives some estimates on the convergence rate.

 \begin{corollary}
     \label{twopointcor}
 Let $S$ be as above. Then, for any $n\geq 1$, one has $\PP(S_{2n-1}=0)=1/2$ and 
  \begin{equation}
      \label{Z2estP2n0}
      e^{-(1-\alpha)n}\alpha^n \leq 2\PP(S_{2n}=0)-1 \leq \alpha^n.
  \end{equation}
  In particular, for any $n\geq 1$,
  $$
  \lim_{\alpha\to 0+}\frac{\log (2\PP(S_{2n}=0)-1) }{\log \alpha} =n, \quad  \lim_{\alpha\to 1-}\frac{\log (1-\PP(S_{2n}=0)) }{\log (1-\alpha)} =1.
  $$
 \end{corollary}
 \begin{remark}
    In the setting of Corollary \ref{twopointcor}, for fixed $\alpha \in (0,1)$, the distribution of $S_n$ converges to the uniform distribution exponentially fast in $n$. The same holds for the usual SRRW on any finite group when $P_\mu$ is irreducible and aperiodic, see the companion paper \cite{peres2026mixing} for more details.
 \end{remark}

We close the introduction by indicating what is new here relative to the companion paper \cite{peres2026mixing}, where the forest representation and the evolving-set method were developed for the usual SRRW on finite groups. First, the estimates below are uniform over transformations $(T_n)$ that may depend on one another and on the past of the forest, and that need neither be linear nor preserve the generating set. Second, on $\RR^d$ the transition-probability bounds require no moment assumption on $\mu$, and the conditional estimates of Proposition \ref{remconditionalestimates} hold from an arbitrary past and at stopping times. Third, the exponential decay of the ERW in Corollary \ref{nonamentran}(ii) holds on every nonamenable Cayley graph and for every memory parameter $p<1$.

\section{Proof of main results}

\textbf{Notation.} For a positive integer $n$, we write $[n]:=\{1,2,\dots,n\}$. We let $C(a_1, a_2,..., a_k)$ denote a positive constant depending only on variables $a_1, a_2,..., a_k$. The actual values of these constants may vary from line to line. When $G$ is countable, we write $L^2(G)=\ell^2(G)$ for the real Hilbert space of square-summable functions $f: G \to \RR$ with respect to counting measure. Its norm and inner product are
$$
\|f\|_2^2:=\sum_{x \in G} f(x)^2 \quad \text { and } \quad\langle f, h\rangle:=\sum_{x \in G} f(x) h(x).
$$
The operator norm of a linear operator $T: L^2(G) \to  L^2(G)$ is defined by 
$$
\|T\|:=\sup _{f \in L^2(G):\|f\|_2=1}\|T f\|_2.
$$

\subsection{Forest representation and conditional estimates}

We relate the generalized SRRW to the Bernoulli percolation on a random recursive tree, as we do for the usual SRRW in \cite[Proposition 2.1]{peres2026mixing}. We note that such a connection was initially observed by K\"ursten \cite{MR3652690} in the setting of elephant random walks.

Using the random variables in Definition \ref{defSRRWG}, we construct a growing random forest $(\mathscr{F}_n)_{n\geq 1}$ and assign a $G$-valued random variable to each node: At time $n=1$, there is a vertex with label 1. We denote by $\mathscr{F}_1$ the forest with this single vertex. Later, at each time step $n\geq 2$:
\begin{enumerate}[(i)] 
  \item  We add and connect a new vertex labeled $n$ to the vertex labeled $u_{n}$ in  $\mathscr{F}_{n-1}$.
 \item  If $\xi_n=0$, the edge connecting the new vertex to the existing vertex is deleted; and if $\xi_n=1$, the edge is retained. We then get a forest with $n$ vertices, which we denote by $\mathscr{F}_n$. 
 \item In each connected component of $\mathscr{F}_n$, we designate the vertex with the smallest label as the root. For $j\in [n]$, we denote by $\CalC_{j, n}$ the cluster rooted at $j$ and denote by $\left|\CalC_{j, n}\right|$ its size, with the convention that $\CalC_{j, n}=\emptyset$ if there is no cluster rooted at $j$. To each non-empty cluster $\CalC_{j,n}$, we assign $g_j$ to the root $j$. The values at the other vertices are determined by $(u_n)_{n\geq 2}$ and $(T_n)_{n\geq 2}$ recursively: If we have assigned $g$ to some vertex $i$ and $\ell$ is joined to $i=u_{\ell}$ by a retained edge, then the value assigned to $\ell$ is $T_{\ell}(g)$.
\end{enumerate}

Note that, for any $n\geq j\geq 1$, the component $\CalC_{j, n}\neq \emptyset$ if and only if $\xi_j=0$ (with the convention that $\xi_1\equiv 0$). In particular, the root of $\CalC_{j,n}$ and the value assigned to the root of $\CalC_{j,n}$ do not change as $n$ increases. Moreover, for fixed $n\geq 2$, one can also obtain $\F_n$ as follows: Construct a random recursive tree by connecting $j$ to $u_j$ for $j= 2,3,\dots,n$; and then perform a Bernoulli percolation on the tree by deleting all edges $(j,u_j)$ with $\xi_j=0$.

The value assigned to vertex $k$ is $X_k$ $(k\geq1)$ from Definition \ref{defSRRWG}. More specifically, if the vertex $k$ belongs to $\CalC_{j,k}$, then
\begin{equation}
    \label{XkdefTrec}
    X_k:=g_j, \ \text{ if }k=j;\quad X_k:= T_k\circ T_{u_k} \circ T_{u_{u_k}} \circ \dots \circ T_{\ell}(g_j), \ \text{ if } k>j,
\end{equation}
where $(k, u_k,u_{u_k},\dots,\ell, j)$ is the unique path in $\mathscr{F}_k$ connecting $k$ and $j$. 

The following Proposition \ref{consgSRRWRRT} shows that one can obtain a generalized SRRW by multiplying those values in order, see Figure \ref{forest7} for an illustration. 

 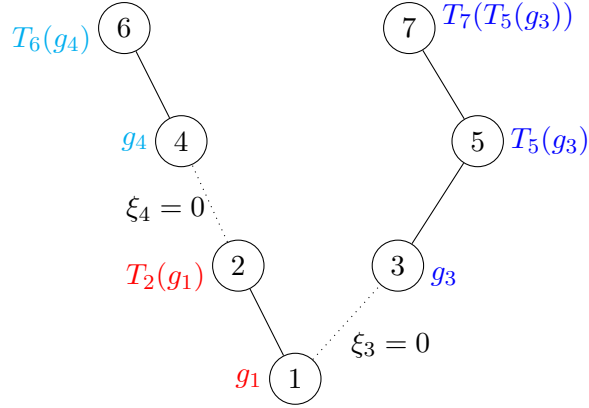
\begin{figure}
    \centering
    \begin{tikzpicture}[scale=1.5]
		\node[circle,draw, minimum size=0.5cm] (A) at  (0,0) {1};
		\node[circle,draw, minimum size=0.5cm] (B) at  (-0.5,1)  {2};
		\node[circle,draw, minimum size=0.5cm] (C) at  (0.9,1)  {3};
		\node[circle,draw, minimum size=0.5cm] (D) at  (-1,2.1)  {4};
		\node[circle,draw, minimum size=0.5cm] (E) at  (1.6,2.1)  {5};
		\node[circle,draw, minimum size=0.5cm] (F) at  (-1.5,3.1)  {6};
		\node[circle,draw, minimum size=0.5cm] (G) at  (1,3.1)  {7}; 
		\draw (A) -- (B);
		\draw[dotted] (A) -- (C);
		\draw[dotted] (B) -- (D);
		\draw (C) -- (E);
		\draw (D) -- (F);
		\draw (E) -- (G);
		\node[left] at (-0.7,1.5) {$\xi_4=0$};
		\node[right] at (0.4,0.3) {$\xi_3=0$};
        \node[left,red] at (-0.2,0) {$g_1$};
        \node[left,red] at (-0.7,0.9) {$T_2(g_1)$};
        \node[left,cyan] at (-1.2,2.1) {$g_4$};
        \node[left,cyan] at (-1.7,3) {$T_6(g_4)$};
        \node[right,blue] at (1.1,0.9) {$g_3$};
        \node[right,blue] at (1.8,2.1) {$T_5(g_3)$};
        \node[right,blue] at (1.2,3.2) {$T_7(T_5(g_3))$};
		\end{tikzpicture}
    \caption{An illustration of $S_7$ and the forest $\mathscr{F}_7$ where $u_2=u_3=1$, $u_4=2$, $u_5=3$, $u_6=4$, $u_7=5$ and $S_7=g_1\cdot T_2(g_1) \cdot g_3 \cdot g_4\cdot T_5(g_3) \cdot T_6(g_4) \cdot T_7(T_5(g_3))$.}
    \label{forest7}
\end{figure}

\begin{proposition}
\label{consgSRRWRRT}
Let $(\mathscr{F}_n)_{n\geq 1}$ and $(X_k)_{k\geq 1}$ be as defined above. Define a random walk $S=(S_n)_{n\in \NN}$ on $G$ by $S_0:=e_G$ and 
\begin{equation}
    \label{perrrtSf}
     S_n := X_1\cdot X_{2}\cdots X_{n}, \quad n\geq 1.
\end{equation}
 Then $S$ is a generalized SRRW with reinforcement parameter $\alpha$, step distribution $\mu$ and transformations $(T_n)_{n\geq 2}$.
\end{proposition}
\begin{remark}
 Proposition \ref{consgSRRWRRT} has already been established by Portillo del Valle for random walk with echoed steps, see \cite[Section 4]{del2025random}.
\end{remark}
\begin{proof}
A root $k$ has value $g_k$, and each non-root vertex $k$ has value $T_k(X_{u_k})$. Thus the assigned values satisfy the recursion in Definition \ref{defSRRWG} with the same random variables.
\end{proof}

For $n\geq 1$, let $\mathscr{I}_n:=\{1\leq j\leq n:|\CalC_{j,n}|=1\}$ be the set of isolated vertices in $\F_n$. Define
\begin{equation}
\label{defconditioning}
\mathcal{K}_n:=\sigma\bigl(\F_n,(T_j)_{2\leq j\leq n},
(g_j)_{j\in[n]\setminus\mathscr{I}_n}\bigr).
\end{equation}
The set $\mathscr I_n$ is determined by the forest, which, together with the transformations, is independent of the entire sequence $(g_j)$. Revealing the labels outside $\mathscr I_n$ therefore leaves the remaining labels independent with law $\mu$, conditionally on $\mathcal K_n$. Every $X_j$ with $j\notin\mathscr{I}_n$ is $\mathcal{K}_n$-measurable, since its component contains no isolated vertex. For each fixed $n$, the kernels below depend on $n$. We suppress this dependence in the notation. Conditionally on $\mathcal{K}_n$, the walk up to time $n$ is a time-inhomogeneous Markov chain with kernels
\begin{equation}
\label{defPj}
P_j=\begin{cases}
P_\mu,&j\in\mathscr{I}_n,\\
P^{(X_j)},&j\notin\mathscr{I}_n,
\end{cases}
\end{equation}
where, for $g\in G$,
\begin{equation}
\label{defPgmatrix}
P^{(g)}(x,y):=\mathbf{1}_{\{y=x\cdot g\}}.
\end{equation}
In particular, the deterministic steps may take values outside $\Gamma=\operatorname{supp}\mu$.

For $0\leq k\leq\ell\leq n$, define
\begin{equation}
\label{defPkell}
P_{k,\ell}:=P_{k+1}\cdots P_\ell,
\qquad P_{k,k}(x,y):=\mathbf{1}_{\{x=y\}}.
\end{equation}
These kernels give the conditional transition probabilities of $S$. When $G$ is countable, we identify them with their Markov operators on $L^2(G)$:
\begin{equation}
\label{Pkelloper}
P_{k,\ell}f(x):=\sum_{y\in G}P_{k,\ell}(x,y)f(y).
\end{equation}
In particular,
\begin{equation}
\label{defPkellxy}
P_{k,\ell}(x,y)=\langle\delta_x,P_{k+1}\cdots P_\ell\delta_y\rangle.
\end{equation}
Here $\delta_z$ is the Kronecker delta at $z$. For $G=\RR^d$, the same construction uses transition kernels on Borel sets instead of matrices.

We let $I(n):=|\mathscr{I}_n|$ be the number of isolated vertices in $\F_n$. The law of the forest depends only on $\alpha$. In \cite[Proposition 3.1]{peres2026mixing}, it has been shown that for any $\alpha \in [0,1)$ and $n\geq 1$, one has 
  \begin{equation}
      \label{upbdprobisolinear}
    \PP\left(I(n) \leq  \frac{(1-\alpha)n}{8}\right) \leq 2e^{-\frac{(1-\alpha)n}{18}}.
  \end{equation}

The following comparison gives a conditional version of (\ref{upbdprobisolinear}). Let $\Theta$ denote the sequence of transformations in Definition \ref{defSRRWG}, or the auxiliary random element in Remark \ref{remforestdependent}, and set
\[
\mathcal A_m:=\sigma\bigl(\Theta,(g_j)_{1\leq j\leq m},
(\xi_j,u_j)_{2\leq j\leq m}\bigr),\qquad m\geq 0.
\]
In either case, $\mathcal H_m:=\sigma(X_1,\ldots,X_m)\subseteq\mathcal A_m$, where $\mathcal H_0$ is trivial. Here $\mathcal A_m$ contains the full driving history and $\mathcal H_m$ records only the steps. The conditioning on $\mathcal K_n$ leaves the isolated labels unrevealed.

\begin{lemma}
\label{lemFmniso}
For integers $m\geq 0$ and $n\geq 1$, write
\[
J_{m,n}:=\bigl|\mathscr I_{m+n}\cap\{m+1,\ldots,m+n\}\bigr|.
\]
Then, for every $K\in\mathbb R$,
\begin{equation}
\label{eqblockisodom}
\PP(J_{m,n}\leq K\mid\mathcal A_m)
\leq \PP(I(n)\leq K+1)\qquad\text{a.s.}
\end{equation}
Moreover,
\begin{equation}
\label{eqblockisotail}
\PP\left(J_{m,n}\leq\frac{(1-\alpha)n}{8}
\,\middle|\,\mathcal A_m\right)
\leq 2e^{-(1-\alpha)n/18}\qquad\text{a.s.}
\end{equation}
Both estimates hold with $\mathcal H_m$ in place of $\mathcal A_m$. They also hold with $m$ replaced by any almost surely finite stopping time for the corresponding filtration.
\end{lemma}

\begin{proof}
The case $m=0$ follows from $J_{0,n}=I(n)$. For $m\geq 1$, we construct a comparison forest on $\{1,\ldots,n\}$ using the vertices in the future block. Its first vertex is a root. For $2\leq j\leq n$, let $v_j$ be uniform on $\{1,\ldots,j-1\}$, independently of all other variables, and set
\[
\widetilde u_j:=
\begin{cases}
u_{m+j}-m,&u_{m+j}>m,\\
v_j,&u_{m+j}\leq m.
\end{cases}
\]
Join $j$ to $\widetilde u_j$ precisely when $\xi_{m+j}=1$. For each $1\leq i<j$,
\[
\PP(\widetilde u_j=i)
=\frac{1}{m+j-1}+\frac{m}{m+j-1}\frac{1}{j-1}
=\frac{1}{j-1}.
\]
These parent choices and the retained-edge indicators are independent of one another and of $\mathcal A_m$. Thus the comparison forest has the law of $\mathscr F_n$, independently of $\mathcal A_m$.

Every isolated vertex $j\geq 2$ in the comparison forest corresponds to an isolated vertex $m+j$ in $\mathscr F_{m+n}$. Indeed, its isolation forces $\xi_{m+j}=0$, and any retained edge from a later vertex to $m+j$ would also appear in the comparison forest. The only possible exception is the first vertex, since $m+1$ may be connected to a vertex at or before time $m$. Consequently, if $\widetilde I(n)$ is the number of isolated vertices in the comparison forest, then
\[
J_{m,n}\geq\widetilde I(n)-1.
\]
This proves (\ref{eqblockisodom}).

To prove (\ref{eqblockisotail}), suppose first that $n\geq2$ and consider the comparison tree before percolation, with parent choices $(\widetilde u_j)_{2\leq j\leq n}$. Let $V_n$ be its number of leaves, where a leaf is a vertex with no children. The root is not a leaf. If a leaf $j$ has $\xi_{m+j}=0$, then it is isolated in the comparison forest, so $m+j$ is isolated in $\mathscr F_{m+n}$. Consequently, if $Z_n$ counts these leaves with deleted parent edges, then
\[
J_{m,n}\geq Z_n,\qquad
\mathcal L\bigl(Z_n\mid(\widetilde u_j)_{2\leq j\leq n},\mathcal A_m\bigr)
=\operatorname{Bin}(V_n,1-\alpha).
\]
The comparison tree is a random recursive tree independent of $\mathcal A_m$. We can therefore apply the leaf estimate and the conditional binomial bound from the proof of \cite[Proposition 3.1]{peres2026mixing}, which give
\[
\PP\left(J_{m,n}\leq\frac{(1-\alpha)n}{8}
\,\middle|\,\mathcal A_m\right)
\leq e^{-n/18}+e^{-25(1-\alpha)n/384}
\leq 2e^{-(1-\alpha)n/18}.
\]
For $n=1$, the last bound exceeds $1$, so (\ref{eqblockisotail}) also holds.

Conditioning again gives the assertions for $\mathcal H_m$. The stopping-time statements follow by applying the deterministic-time estimates on each event $\{\tau=m\}$ and summing over $m$.
\end{proof}

\begin{proposition}[Conditional estimates]
\label{remconditionalestimates}
Let $S$ be a generalized SRRW as in Definition \ref{defSRRWG} or Remark \ref{remforestdependent}, with reinforcement parameter $\alpha\in[0,1)$ and step distribution $\mu$. Fix integers $m\geq0$ and $n\geq1$, and write $J_{m,n}$ as in Lemma \ref{lemFmniso}. Let $\mathcal G_{m,n}$ be the $\sigma$-algebra generated by $\mathcal A_m$, the pairs $(\xi_j,u_j)_{2\leq j\leq m+n}$, and the labels $g_j$ of the non-isolated vertices in $\{m+1,\ldots,m+n\}$.
\begin{enumerate}[(i)]
\item Conditionally on $\mathcal G_{m,n}$, the labels at the $J_{m,n}$ isolated vertices of the block are independent with law $\mu$, all other steps in the block are $\mathcal G_{m,n}$-measurable, and, starting from $S_m$, the conditional law of the walk on the block is that of a time-inhomogeneous Markov chain with kernels as in (\ref{defPj}): $P_\mu$ at the isolated times of the block and translation kernels at all remaining times.
\item If $G$ is countably infinite and $q:=\|P_\mu\|_{L^2(G)\to L^2(G)}<1$, then
\begin{equation}
\label{eqconditionalheatkernel}
\sup_{x\in G}\PP(S_{m+n}=x\mid\mathcal A_m)
\leq q^{(1-\alpha)n/8}+2e^{-(1-\alpha)n/18}.
\end{equation}
\item If $G$ is finite of size $N$, $\pi$ is the uniform distribution on $G$, and $\|P_\mu|_{L^2_0(G)}\|\leq q_0$ for some $q_0\in(0,1)$, where $L^2_0(G)$ is the subspace of $L^2(G)$ of functions with mean zero, then
\begin{equation}
\label{eqconditionalmixing}
\bigl\|\mathcal L(S_{m+n}\mid\mathcal A_m)-\pi\bigr\|_{\mathrm{TV}}
\leq \frac12\sqrt{N-1}\,q_0^{(1-\alpha)n/8}
+2e^{-(1-\alpha)n/18}.
\end{equation}
\end{enumerate}
All three assertions remain valid when $m$ is replaced by any almost surely finite stopping time for $(\mathcal A_m)_{m\geq0}$. The bounds in (ii) and (iii) also hold with $\mathcal H_m$ in place of $\mathcal A_m$, including at almost surely finite stopping times for $(\mathcal H_m)_{m\geq0}$.
\end{proposition}

\begin{proof}
By the definition of $\Theta$, the transformations up to time $m+n$ are $\mathcal G_{m,n}$-measurable, and so is the forest $\mathscr F_{m+n}$. Every step $X_k$ with $m+1\leq k\leq m+n$ and $k\notin\mathscr I_{m+n}$ is then $\mathcal G_{m,n}$-measurable: its value is obtained from the label at the root of its cluster by composing the transformations along the retained-edge path, and the root either lies at a time at most $m$, where its label is $\mathcal A_m$-measurable, or is a non-isolated vertex of the block, whose label has been revealed. Conditionally on $\Theta$ and the pairs $(\xi_j,u_j)_{2\leq j\leq m+n}$, the labels are independent with law $\mu$. Revealing the past labels and the labels at non-isolated vertices of the block therefore leaves the remaining labels independent with law $\mu$. This proves (i).

For (ii), conditionally on $\mathcal G_{m,n}$ the law of $S_{m+n}$ is $\delta_{S_m}$ times the kernel product in (i). Each translation kernel has operator norm $1$, so this product has norm at most $q^{J_{m,n}}$, and
$$
\PP(S_{m+n}=x\mid\mathcal A_m)
=\EE\bigl[\PP(S_{m+n}=x\mid\mathcal G_{m,n})\bigm|\mathcal A_m\bigr]
\leq\EE\bigl(q^{J_{m,n}}\mid\mathcal A_m\bigr).
$$
Splitting according to $J_{m,n}\leq(1-\alpha)n/8$ and using (\ref{eqblockisotail}) proves (\ref{eqconditionalheatkernel}).

For (iii), every factor in the kernel product preserves $\pi$, the translations are isometries of $L^2_0(G)$, and $P_\mu$ has norm at most $q_0$ there. The Cauchy--Schwarz inequality gives
\[
\begin{aligned}
\bigl\|\mathcal L(S_{m+n}\mid\mathcal G_{m,n})-\pi\bigr\|_{\mathrm{TV}}
&\leq\frac{\sqrt N}{2}\,q_0^{J_{m,n}}\|\delta_{S_m}-\pi\|_2\\
&=\frac12\sqrt{N-1}\,q_0^{J_{m,n}}.
\end{aligned}
\]
By convexity of total variation, taking conditional expectations given $\mathcal A_m$ and using this bound when $J_{m,n}>(1-\alpha)n/8$, and the bound $1$ otherwise, proves (\ref{eqconditionalmixing}) by (\ref{eqblockisotail}).

The versions of (ii) and (iii) with $\mathcal H_m$ follow from $\mathcal H_m\subseteq\mathcal A_m$ by the tower property and convexity of total variation. The stopping-time versions follow by applying the deterministic-time assertions on each event $\{\tau=m\}$ and summing over $m$.
\end{proof}

\begin{remark}
\label{remkernelgeneral}
Any bound for the kernel product in Proposition \ref{remconditionalestimates}(i) that is uniform over the translation factors and requires only a specified number of $P_\mu$ factors can be combined with (\ref{eqblockisotail}) to give a corresponding conditional estimate. This also applies to the lazy heat-kernel estimates proved below.
\end{remark}

\subsection{Concentration estimates and class-function comparison}

We now prove Proposition \ref{Strandabove3} by using (\ref{upbdprobisolinear}) and classical results on the concentration function.

\begin{proof}[Proof of Proposition \ref{Strandabove3}]
Since $\RR^d$ is abelian, we can write
\begin{equation}
\label{SndecomRd}
S_n=\sum_{j\in\mathscr{I}_n}g_j+
\sum_{j\in[n]\setminus\mathscr{I}_n}X_j.
\end{equation}
Conditionally on $\mathcal{K}_n$, the second sum is fixed and the first is a sum of $I(n)$ independent variables with law $\mu$. Since $\mu$ is non-singular, Esseen's concentration bound \cite[Theorem 6.2 and its corollary]{MR231419} gives
\[
\sup_{x\in\RR^d}\PP(S_n\in B(x,r)\mid\mathcal{K}_n)
\leq \frac{C(r,\mu)}{(I(n)+1)^{d/2}}.
\]
The constant can be chosen to include $I(n)=0$. Taking expectations and using (\ref{upbdprobisolinear}), we obtain
\[
\sup_{x\in\RR^d}\PP(S_n\in B(x,r))
\leq C(r,\mu)\left(\frac{8}{1-\alpha}\right)^{d/2}n^{-d/2}
+2e^{-(1-\alpha)n/18}.
\]
This proves the bound. If $d\geq 3$, the Borel--Cantelli lemma, applied to balls of positive integer radius, gives $\|S_n\|\to\infty$ almost surely.
\end{proof}

When $\mu$ is a class function, we can also group the ``free" steps (namely, the i.i.d. $\mu$-distributed steps corresponding to isolated vertices) together, as in the proof of Proposition \ref{Strandabove3}. We can then upper bound the transition probabilities when there is a sufficient number of free steps.

\begin{proof}[Proof of Proposition \ref{propclassbd}]
For every $g\in G$ and $f\in L^2(G)$, the class-function assumption gives
\[
\begin{aligned}
(P_\mu P^{(g)}f)(x)
&=\sum_{y\in G}\mu(x^{-1}yg^{-1})f(y)\\
&=\sum_{y\in G}\mu(g^{-1}x^{-1}y)f(y)
=(P^{(g)}P_\mu f)(x).
\end{aligned}
\]
Thus
\begin{equation}
\label{classinterchmug}
P_\mu P^{(g)}=P^{(g)}P_\mu.
\end{equation}
Condition on $\mathcal{K}_n$, and let $m_1<\cdots<m_{n-I(n)}$ be the non-isolated vertices. Commuting the $P_\mu$ factors past the translation operators gives
\[
\PP^{(\alpha)}(S_n=x\mid\mathcal{K}_n)
=P_\mu^{I(n)}(X_{m_1}\cdots X_{m_{n-I(n)}},x).
\]
An empty product here equals $e_G$. By translation invariance and the assumption on the non-reinforced walk, this is at most $\varepsilon$ whenever $I(n)\geq m$. If $n\geq8m/(1-\alpha)$, then (\ref{upbdprobisolinear}) implies
\[
\PP^{(\alpha)}(S_n=x)
\leq\varepsilon+\PP(I(n)<m)
\leq\varepsilon+2e^{-(1-\alpha)n/18}. 
\]
\end{proof}

\subsection{Nonamenable groups and elephant random walks}

We next turn to Corollary \ref{nonamentran}. We first record the representation of the ERW as a generalized SRRW.

\begin{lemma}
\label{ERWgSRRW}
The ERW on a Cayley graph $G_{\Gamma}$ with a finite symmetric generating set $\Gamma$ of size at least two is a generalized SRRW. If its memory parameter satisfies $p\in(0,1)$, the reinforcement parameter in this representation is strictly less than $1$.
\end{lemma}
\begin{proof}
Write $\Gamma=\{s_1,\ldots,s_r\}$, where $r\geq 2$, and let $\mu$ be uniform on $\Gamma$. If $p\geq 1/r$, the ERW is a usual SRRW with parameter $\alpha=(rp-1)/(r-1)$, see \cite[Section 2.2]{mukherjee2025elephant}. If $p<1/r$, let $\sigma=(1\,2\,\cdots\,r)$ and set
\[
T_n(s_i)=s_{\sigma^{Y_n}(i)},\qquad 1\leq i\leq r,
\]
where the $Y_n$ are i.i.d. uniform on $\{1,\ldots,r-1\}$ and independent of the innovations and the variables $(\xi_n,u_n)$. Extend $T_n$ by the identity outside $\Gamma$. Then $T_n(s_i)$ is uniform on $\Gamma\setminus\{s_i\}$. With $\alpha=1-rp$, conditionally on selecting the past direction $s_i$, the probabilities of choosing $s_i$ and any fixed $s_j\neq s_i$ are precisely the ERW probabilities:
\[
\frac{1-\alpha}{r}=p,
\qquad
\frac{1-\alpha}{r}+\frac{\alpha}{r-1}=\frac{1-p}{r-1}.\qedhere
\]
\end{proof}

The representation in Lemma \ref{ERWgSRRW} has parameter $\alpha=1$ when $p=0$. To include this endpoint, we use the following estimates for the numbers of steps in each direction. These estimates depend only on the step sequence, and not on the group structure.

\begin{lemma}
\label{lemERWcounts}
Let $(X_n)_{n\geq 1}$ be the step sequence of an ERW in $r\geq 2$ directions $\{s_1,\ldots,s_r\}$, with memory parameter $p\in[0,1)$. Set
\begin{equation}
\label{countstepS}
N_n(i):=\sum_{j=1}^n\mathbf{1}_{\{X_j=s_i\}},\qquad
\Delta_n(i):=N_n(i)-\frac nr,\qquad
\beta:=\frac{rp-1}{r-1},
\end{equation}
and write $\Delta_n=(\Delta_n(i))_{1\leq i\leq r}$. Define
\[
V_\beta(n):=
\begin{cases}
n,&\beta<1/2,\\
n\log(n+1),&\beta=1/2,\\
n^{2\beta},&\beta>1/2.
\end{cases}
\]
There exist constants $c,C>0$, depending only on $p$ and $r$, such that, for all $n\geq 1$ and $u\geq 0$,
\begin{equation}
\label{eqERWcounttail}
\PP(\|\Delta_n\|\geq u)\leq C\exp\left(-\frac{cu^2}{V_\beta(n)}\right).
\end{equation}
In particular, $\EE\|\Delta_n\|^2\leq C V_\beta(n)$. Moreover, for every $\delta>0$, there exist constants $c_\delta,C_\delta>0$, depending only on $p,r,\delta$, such that, for all $m\geq 1$,
\begin{equation}
\label{eqERWcountproportion}
\PP\left(\sup_{n\geq m}\frac{\|\Delta_n\|}{n}\geq\delta\right)
\leq C_\delta e^{-c_\delta m}.
\end{equation}
\end{lemma}
\begin{proof}
Recall that $\mathcal H_n=\sigma(X_1,\ldots,X_n)$. By the definition of the ERW,
\begin{equation}
\label{eqERWdirectionprob}
q_{n+1}(i):=\PP(X_{n+1}=s_i\mid\mathcal H_n)
=\frac1r+\beta\frac{\Delta_n(i)}n.
\end{equation}
Set $\varepsilon_1(i)=\Delta_1(i)$ and, for $n\geq 1$, set
$\varepsilon_{n+1}(i)=\mathbf{1}_{\{X_{n+1}=s_i\}}-q_{n+1}(i)$.
These are martingale differences bounded in absolute value by $1$, and
\begin{equation}
\label{eqERWcountrecursion}
\Delta_{n+1}(i)=\left(1+\frac\beta n\right)\Delta_n(i)+\varepsilon_{n+1}(i).
\end{equation}
Thus
\[
\Delta_n(i)=\sum_{j=1}^n b_{j,n}\varepsilon_j(i),\qquad
b_{j,n}:=\prod_{k=j}^{n-1}\left(1+\frac\beta k\right),
\]
with the convention that an empty product is $1$. Since $-1\leq\beta<1$, these deterministic coefficients are nonnegative. Whenever $b_{j,n}>0$, the inequality $\log(1+x)\leq x$ and comparison of harmonic sums with an integral give
\[
\log b_{j,n}\leq\beta\sum_{k=j}^{n-1}\frac1k
=\beta\log\frac nj+O_\beta(1).
\]
Thus $b_{j,n}\leq C(n/j)^\beta$. This bound also holds for $\beta=-1$ because $b_{1,n}=0$ for $n\geq2$. Thus,
\[
\sum_{j=1}^n b_{j,n}^2
\leq Cn^{2\beta}\sum_{j=1}^n j^{-2\beta}
\leq C\begin{cases}
n,&2\beta<1,\\
n\log(n+1),&2\beta=1,\\
n^{2\beta},&2\beta>1.
\end{cases}
\]
Orthogonality of the martingale differences then gives $\EE\|\Delta_n\|^2\leq C V_\beta(n)$.  For each fixed $n$, the coefficients $(b_{j,n})$ are deterministic, so the Azuma--Hoeffding inequality applies to the martingale sum in each coordinate. Since $\|\Delta_n\|\geq u$ implies $|\Delta_n(i)|\geq u/\sqrt r$ for some $i$, a union bound gives (\ref{eqERWcounttail}). 

When $\beta\geq1/2$, substituting $u=\delta n$ into (\ref{eqERWcounttail}) does not give an exponential bound in $n$. To prove (\ref{eqERWcountproportion}), we instead restart the recursion at a small fixed fraction of $n$. The contribution of the earlier history then becomes uniformly small, while the squared coefficients of the remaining martingale sum add up to $O(1/n)$.

Put $A_n(i)=\Delta_n(i)/n$ and $\gamma=1-\beta\in(0,2]$. Fix $t>0$ and choose $\theta\in(0,1/4)$ such that $(2\theta)^\gamma\leq t/2$. For sufficiently large $n$, let $k=\lceil\theta n\rceil$. Dividing (\ref{eqERWcountrecursion}) by $n+1$ and iterating from $k$ gives
\[
A_n(i)=\prod_{\ell=k+1}^n\left(1-\frac\gamma\ell\right)A_k(i)
+\sum_{j=k+1}^n\frac{\varepsilon_j(i)}j
\prod_{\ell=j+1}^n\left(1-\frac\gamma\ell\right).
\]
All factors in these products lie in $[0,1]$, and
\[
\prod_{\ell=k+1}^n\left(1-\frac\gamma\ell\right)
\leq\exp\left(-\gamma\sum_{\ell=k+1}^n\frac1\ell\right)
\leq\left(\frac{k+1}{n+1}\right)^\gamma.
\]
For all sufficiently large $n$, we have $(k+1)/(n+1)\leq2\theta$. Since $|A_k(i)|\leq1$, the first term therefore has absolute value at most $(2\theta)^\gamma\leq t/2$.
The sum is a martingale sum with deterministic coefficients whose squared sum is at most $n/k^2\leq1/(\theta^2n)$. Another application of the Azuma--Hoeffding inequality yields
\[
\PP(|A_n(i)|\geq t)\leq 2\exp(-t^2\theta^2 n/8).
\]
Increasing the prefactor covers the finitely many remaining $n$. Take $t=\delta/(2\sqrt r)$, sum over the coordinates, and then over $n\geq m$ to obtain (\ref{eqERWcountproportion}).
\end{proof}

\begin{proof}[Proof of Corollary \ref{nonamentran}]
(i). Let $\nu$ be the uniform probability measure on $\Gamma$, and put $r=|\Gamma|$. Since $\nu$ is symmetric and irreducible and $G_\Gamma$ is non-amenable, Kesten's theorem gives $\|P_\nu\|<1$, see, for example, \cite[Theorem 5.1.5]{MR4628024}. Let $b=r\min_{g\in\Gamma}\mu(g)>0$. When $b<1$ we can write $\mu=b\nu+(1-b)\nu'$ for a probability measure $\nu'$, and 
\[
\|P_\mu\|\leq b\|P_\nu\|+1-b<1.
\]
The same conclusion holds when $b=1$, since then $\mu=\nu$.

By (\ref{defPkellxy}), for every $x\in G$,
\[
\PP(S_n=x\mid\mathcal{K}_n)=\langle\delta_{e_G},P_1\cdots P_n\delta_x\rangle\leq\prod_{j=1}^n\|P_j\|\|P_\mu\|^{I(n)},
\]
where $\|\delta_x\|_2=1$ for every $x$, and each non-isolated step gives a translation operator of norm $1$. In particular, for any $\ell\geq1$,
\[
P_\nu^\ell(y,x)=\langle\delta_y,P_\nu^\ell\delta_x\rangle
\leq\|P_\nu\|^\ell.
\]
Taking expectations and using (\ref{upbdprobisolinear}), we obtain
\[
\sup_{x\in G}\PP(S_n=x)
\leq\|P_\mu\|^{(1-\alpha)n/8}
+2e^{-(1-\alpha)n/18}
\leq Ce^{-cn}.
\]
To absorb the prefactor, note that the fresh-step part of the transition law gives, for all $n\geq1$,
\[
\sup_{x\in G}\PP(S_n=x)
\leq q:=\alpha+(1-\alpha)\max_{g\in\Gamma}\mu(g)<1.
\]
Choose $N\geq1$ so that $Ce^{-cn}\leq e^{-cn/2}$ for $n\geq N$. Then (\ref{returnprobexpon}) holds with $\kappa=\max\{e^{-c/2},q^{1/N}\}<1$.

(ii). Write $\Gamma=\{s_1,\ldots,s_r\}$ and let $\Delta_n$ and $\beta$ be as in Lemma \ref{lemERWcounts}. Let $P_\nu$ be the simple random walk operator on $G_\Gamma$, and choose $\eta>0$ such that
\[
a:=(1+\eta)\|P_\nu\|<1.
\]
For $n\geq2$, put $m=\lfloor n/2\rfloor$ and define
\[
E_n:=\left\{\frac{\|\Delta_k\|}{k}\leq\frac\eta r
\text{ for every }m\leq k<n\right\}.
\]
By (\ref{eqERWcountproportion}), $\PP(E_n^c)\leq Ce^{-cn}$. Along every step history contributing to $E_n$, (\ref{eqERWdirectionprob}) and $|\beta|\leq1$ give
\[
q_{k+1}(i)\leq\frac{1+\eta}{r},\qquad m\leq k<n,\quad 1\leq i\leq r.
\]
Condition on $\mathcal H_m$ and sum the original path probabilities over all step sequences leading to $x$ and satisfying $E_n$. Each such probability is at most $(1+\eta)^{n-m}$ times the corresponding simple random walk probability. Therefore,
\[
\PP(S_n=x,E_n\mid\mathcal H_m)
\leq(1+\eta)^{n-m}P_\nu^{n-m}(S_m,x)
\leq a^{n-m}.
\]
It follows that $\sup_x\PP(S_n=x)\leq Ce^{-cn}$. A finite symmetric generating set of a nonamenable group has $r\geq3$. Put
\[
M:=\max\left\{p,\frac{1-p}{r-1}\right\}<1.
\]
For $n\geq2$, the conditional probability $\PP(X_n=s_i\mid\mathcal H_{n-1})$ is a convex combination of $p$ and $(1-p)/(r-1)$, with weights $N_{n-1}(i)/(n-1)$ and $1-N_{n-1}(i)/(n-1)$. Consequently,
\[
\begin{aligned}
\PP(S_n=x\mid\mathcal H_{n-1})
&=\PP(X_n=S_{n-1}^{-1}x\mid\mathcal H_{n-1})\leq M.
\end{aligned}
\]
The probability is zero if $S_{n-1}^{-1}x\notin\Gamma$. For $n=1$, the same bound follows from $1/r\leq M$. Taking expectations gives $\sup_x\PP(S_n=x)\leq M$ for all $n\geq1$. As in Part (i), the prefactor can now be absorbed into a constant $\rho\in(0,1)$.

Finally, either exponential bound implies positive lower speed. Indeed, $|\{x\in G:d(e_G,x)\leq k\}|\leq(r+1)^k$ for integers $k\geq0$. If $\lambda\in(0,1)$ denotes $\kappa$ or $\rho$, choose $\varepsilon>0$ such that $(r+1)^\varepsilon\lambda<1$. Then
\[
\sum_{n=1}^\infty\PP(d(e_G,S_n)\leq\varepsilon n)
\leq\sum_{n=1}^\infty\bigl((r+1)^\varepsilon\lambda\bigr)^n<\infty.
\]
The Borel--Cantelli lemma gives $\liminf_n d(e_G,S_n)/n\geq\varepsilon$ almost surely in both cases.
\end{proof}

\subsection{Evolving sets}
\label{secevosets}

To prove Theorem \ref{lazyGinf}, we adapt the evolving set method of Morris and Peres \cite{MR2198701}. This method was also used in the companion paper \cite{peres2026mixing} to estimate the mixing times of the SRRW on finite groups.

Fix $n\geq 1$ and condition on $\mathcal{K}_n$ defined in (\ref{defconditioning}). Recall the kernels $(P_j)_{j\in[n]}$ and $(P_{k,\ell})_{0\leq k\leq\ell\leq n}$ from (\ref{defPj}) and (\ref{defPkell}). Each $P_j$ is either $P_\mu$ or $P^{(g)}$ for some $g\in G$. Given these kernels, let $(U_j)_{j\in[n]}$ be independent uniform random variables on $(0,1)$ and define an evolving set process by
$$
 W_{j+1}=\left\{y\in G:\sum_{x\in W_j}P_{j+1}(x,y)\geq U_{j+1}\right\},
 \qquad 0\leq j<n.
$$
We write $\mathbf{P}$ and $\mathbf{E}$ for its conditional law and expectation. All initial sets are finite. Since the kernels are doubly stochastic,
$$
 \mathbf{E}(|W_{j+1}|\mid W_j)=|W_j|,
$$
so the sets remain finite a.s. and $(|W_j|)$ is a martingale. By \cite[Lemma 4.3 (i)]{peres2026mixing}, for $0\leq k\leq\ell\leq n$,
$$
 P_{k,\ell}(x,y)=\mathbf{P}(y\in W_\ell\mid W_k=\{x\}).
$$

\begin{lemma}
\label{distevolset}
For any $0\leq k\leq\ell\leq n$ and $x\in G$,
$$
 \left(\sum_{y\in G}P_{k,\ell}(x,y)^2\right)^{1/2}
 \leq\mathbf{E}\left(\sqrt{|W_\ell|}\mid W_k=\{x\}\right).
$$
\end{lemma}
\begin{proof}
This was proved in \cite[Equation (39)]{MR2198701} (with the invariant measure $\pi$ there being the counting measure). One can also apply Minkowski's inequality (see e.g. \cite[Section 6.3]{MR1681462}):
$$
\left(\sum_{y \in G} P_{k, \ell}(x, y)^2\right)^{1 / 2}=\left\|\mathbf{E}(\mathbf{1}_{W_\ell}(\cdot)\mid W_k=\{x\})\right\|_{\ell^2(G)} \leq \mathbf{E}_x \left\|\mathbf{1}_{W_{\ell}}\right\|_{\ell^2(G)} =\mathbf{E}\left(\sqrt{|W_\ell|}\mid W_k=\{x\}\right).
$$
\end{proof}

For nonempty finite sets $W,A\subset G$, define the Doob transform
$$
 \widehat K_j(W,A)=\frac{|A|}{|W|}\mathbf{P}(W_j=A\mid W_{j-1}=W).
$$
The martingale property shows that these are transition kernels on the nonempty finite subsets of $G$. Write $\widehat{\mathbf{P}}$ and $\widehat{\mathbf{E}}$ for the corresponding conditional law and expectation. Along a path of nonempty sets, the successive weights $|W_j|/|W_{j-1}|$ telescope to $|W_\ell|/|W_k|$. Paths ending at the empty set receive zero weight. Thus
\begin{equation}
\label{doobtranind}
 \widehat{\mathbf{P}}(W_\ell=A\mid W_k=W)
 =\frac{|A|}{|W|}\mathbf{P}(W_\ell=A\mid W_k=W).
\end{equation}
For a finite set $W\subset G$, let
\begin{equation}
\label{defWmu}
 W_\mu=\left\{y\in G:\sum_{x\in W}P_\mu(x,y)\geq\widetilde U\right\},
\end{equation}
where $\widetilde U$ is uniform on $(0,1)$, and write $K_\mu(W,A)=\mathbf{P}(W_\mu=A)$. For nonempty finite $W$, put
\begin{equation}
\label{defpsiW}
 \psi(W)=1-\mathbf{E}\sqrt{\frac{|W_\mu|}{|W|}},
\end{equation}
and define the nonincreasing root profile by
\begin{equation}
\label{defpsirinfG}
 \psi(r)=\inf\{\psi(W):\varnothing\neq W\subset G,\ |W|\leq r\},
 \qquad r\geq 1.
\end{equation}
Note that $\Phi(r)>0$ for each fixed $r$. Therefore, for countably infinite $G$, if $P_\mu$ is irreducible and $\mu(e_G)>0$, this root profile is positive by (\ref{psiinePhi}) below.

\begin{proposition}
\label{propevopsi}
Assume that $G$ is countably infinite, $P_\mu$ is irreducible and $\mu(e_G)>0$. For any $0\leq k\leq\ell\leq n$, any $x\in G$ and any $\varepsilon \in (0,1)$,
$$
 \sum_{y\in G}P_{k,\ell}(x,y)^2\leq\varepsilon
 \quad\text{if}\quad
 |\mathscr I_n\cap\{k+1,\ldots,\ell\}|
 \geq\int_4^{4/\varepsilon}\frac{du}{u\psi(u)}.
$$
\end{proposition}
\begin{proof}
Start the evolving set process from $W_k=\{x\}$. Under $\widehat{\mathbf{P}}$, all its sets are nonempty. Since $|W_k|=1$, (\ref{doobtranind}) weights the law of $W_\ell$ by $|W_\ell|$. Together with Lemma \ref{distevolset}, this gives
\begin{equation}
\label{PkellnormZell}
 \left(\sum_{y\in G}P_{k,\ell}(x,y)^2\right)^{1/2}
 \leq\mathbf{E}\sqrt{|W_\ell|}
 =\widehat{\mathbf{E}}\frac{1}{\sqrt{|W_\ell|}}.
\end{equation}
Let $j_1<\cdots<j_q$ be the isolated vertices in $\{k+1,\ldots,\ell\}$, and put $j_0=k$. At step $j \in \{k+1,\dots,\ell\}\setminus\{j_1,j_2,\dots,j_q\}$, $P_j=P^{(g)}$ for some $g\in G$, so $W_j=W_{j-1}g$ under both laws. In particular, its size does not change. Set
$$
 Z_m=|W_{j_m}|^{-1/2},\qquad 0\leq m\leq q.
$$
Given $W_{j_{m-1}}$, the set $V_m:=W_{j_m-1}$ is a deterministic translate of $W_{j_{m-1}}$. Thus $|V_m|=Z_{m-1}^{-2}$, and the Doob transform gives
\begin{align*}
 \widehat{\mathbf{E}}\left(\frac{Z_m}{Z_{m-1}}\,\middle|\,W_{j_{m-1}}\right)
 &=\sum_{A\subset G:\,|A|<\infty}\sqrt{\frac{|A|}{|V_m|}}K_\mu(V_m,A)\\
 &=1-\psi(V_m)\leq1-\psi(Z_{m-1}^{-2}).
\end{align*}
Since $Z_0=1$, \cite[Lemma 11 (iii)]{MR2198701}, applied to $f_0(z)=\psi(z^{-2})$ on $(0,1]$, yields
$$
 \widehat{\mathbf{E}}Z_q\leq\sqrt\varepsilon
 \quad\text{if}\quad
 q\geq\int_{\sqrt\varepsilon}^1\frac{2\,dz}{z\psi(4/z^2)}
 =\int_4^{4/\varepsilon}\frac{du}{u\psi(u)}.
$$
Here $f_0$ may be extended by $f_0(0)=0$ and $f_0(z)=\psi(1)$ for $z>1$. Finally, $|W_\ell|=|W_{j_q}|$, so (\ref{PkellnormZell}) proves the claim.
\end{proof}

We also use the reversed kernels
$$
 \bar P_j(x,y)=P_{n+1-j}(y,x),\qquad j\in[n].
$$
Define $\bar P_{k,\ell}:=\bar P_{k+1}\cdots\bar P_\ell$ for $0\leq k\leq\ell\leq n$, with an empty product equal to the identity kernel. These kernels are doubly stochastic and satisfy
$$
 P_{k,\ell}(x,y)=\bar P_{n-\ell,n-k}(y,x).
$$
At the times $\bar{\mathscr I}_n:=\{n+1-j:j\in\mathscr I_n\}$, the reversed kernel is $P_{\check\mu}$, where $\check\mu(g)=\mu(g^{-1})$; otherwise it is a translation. Proposition \ref{propevopsi} therefore applies with the root profile $\bar\psi$ of $P_{\check\mu}$. For every finite $A\subset G$, double stochasticity gives
$$
 P_\mu(A,A^c)=|A|-P_\mu(A,A)=P_\mu(A^c,A)=P_{\check\mu}(A,A^c).
$$
Consequently the two kernels have the same isoperimetric profile. By \cite[Lemma 3]{MR2198701},
\begin{equation}
\label{psiinePhi}
 \min\{\psi(r),\bar\psi(r)\}
 \geq\frac{\mu_0^2\Phi(r)^2}{2(1-\mu_0)^2},\qquad r\geq1.
\end{equation}

\begin{proof}[Proof of Theorem \ref{lazyGinf}]
Fix $\varepsilon\in(0,1)$ and write
$$
 H_\varepsilon=\int_4^{8/\varepsilon}\frac{du}{u\Phi(u)^2},
 \qquad
 J_\varepsilon=\frac{2(1-\mu_0)^2}{\mu_0^2}H_\varepsilon.
$$
On the event $\{I(n)\geq1+2J_\varepsilon\}$, let $m$ be the time of the $\lceil J_\varepsilon\rceil$-th isolated vertex. Since $\lceil J_\varepsilon\rceil\leq J_\varepsilon+1$, we have $I(n)-\lceil J_\varepsilon\rceil\geq J_\varepsilon$. Thus $m<n$ and both $[m]$ and $\{m+1,\ldots,n\}$ contain at least $J_\varepsilon$ isolated vertices. Proposition \ref{propevopsi} and (\ref{psiinePhi}) imply that
$$
 \sum_{z\in G}P_{0,m}(x,z)^2\leq\frac\varepsilon2,
 \qquad
 \sum_{z\in G}P_{m,n}(z,y)^2
 =\sum_{z\in G}\bar P_{0,n-m}(y,z)^2\leq\frac\varepsilon2.
$$
By Cauchy--Schwarz, $P_{0,n}(x,y)\leq\varepsilon/2$ on this event. Hence
$$
 \PP(S_n=y)\leq\frac\varepsilon2+
 \PP\bigl(I(n)<1+2J_\varepsilon\bigr).
$$
By (\ref{upbdprobisolinear}), the last term is at most $\varepsilon/2$ whenever
$$
 n\geq\frac{1}{1-\alpha}
 \max\left\{8(1+2J_\varepsilon),\ 18\log\frac{4}{\varepsilon}\right\}.
$$
The first term ensures $1+2J_\varepsilon\leq(1-\alpha)n/8$, and the second gives $2e^{-(1-\alpha)n/18}\leq\varepsilon/2$. Since $\Phi(u)\leq1$, we have $H_\varepsilon\geq\log(2/\varepsilon)$, and hence
\[
1+\log(4/\varepsilon)\leq C H_\varepsilon,
\qquad 0<\varepsilon<1.
\]
Both terms in the maximum are therefore bounded by $C(\mu_0)H_\varepsilon$, which gives the stated condition.
\end{proof}

\begin{proof}[Proof of Corollary \ref{corlazytran}]
Let $\mu_s=(\mu+\check\mu)/2$ and $\mu_*:=\min_{g\in\Gamma}\mu(g)$. The measure $\mu_s$ has support $\Gamma\cup\Gamma^{-1}$, its positive masses are at least $\mu_*/2$, and
$P_{\mu_s}(A,A^c)=P_\mu(A,A^c)$ for every finite $A\subset G$. The isoperimetric inequality of Coulhon and Saloff-Coste \cite{MR1232845} therefore gives, for every nonempty finite $A$,
\begin{equation}
\label{lowbdphiAi}
 \Phi(A)\geq\frac{\mu_*}{2}\frac{|\partial_E A|}{|A|}
 \geq\frac{\mu_*}{4R(2|A|)},
\end{equation}
where $\partial_E A$ is the edge boundary in $G_\Gamma$, and $R(m)$ is the smallest radius of a ball containing at least $m$ vertices.

In case (i), this gives $\Phi(r)\geq c r^{-1/d}$ for $r\geq1$, where $c>0$ depends on $G$ and $\mu$. Therefore
$$
 \int_4^{8/\varepsilon}\frac{du}{u\Phi(u)^2}
 \leq\frac{1}{c^2}\int_4^{8/\varepsilon}u^{2/d-1}\,du
 =\frac{d}{2c^2}\left[\left(\frac8\varepsilon\right)^{2/d}-4^{2/d}\right]
 \leq C\varepsilon^{-2/d}.
$$
Choose $\varepsilon=C'n^{-d/2}$ for large $n$. The right-hand side of the time condition in Theorem \ref{lazyGinf} is at most $C''(C')^{-2/d}n$, so taking $C'$ sufficiently large makes that condition hold. Increasing $C_1$ handles the remaining values of $n$. The same argument applies whenever the volume growth is bounded below by a constant times $r^d$, so at least cubic growth implies transience.

In case (ii), (\ref{lowbdphiAi}) gives $\Phi(r)\geq c/\log(2r)$, and hence the integral in Theorem \ref{lazyGinf} is at most $C[\log(2/\varepsilon)]^3$. Choosing $\varepsilon=2e^{-c'n^{1/3}}$ for large $n$ proves (ii). In case (iii), $\inf_{r\geq1}\Phi(r)>0$, so the same integral is at most $C\log(2/\varepsilon)$. This proves (iii).
\end{proof}

\subsection{Elephant polynomials}

To prove Corollary \ref{twopointcor}, we use the elephant polynomials introduced by Gu\'erin, Laulin and Raschel \cite{MR4897955}. They satisfy
\begin{equation}
  \label{defelepoly}
  \left\{\begin{aligned}
R_1(x)&:=x, \\
R_{n+1}(x)&:=xR_n(x)-\frac{\alpha}{n}(1-x^2)R_n'(x), \qquad n\geq 1,
\end{aligned}\right.
\end{equation}
where $\alpha\in\RR$. Let $S^{(E)}$ be the ERW on $\ZZ$ with memory parameter $p\in[0,1]$ and first step uniform on $\{-1,1\}$. Its characteristic function satisfies
\begin{equation}
    \label{charERWpoly}
    \varphi_n^{(E)}(t)=R_n(\cos t), \qquad t\in\RR,
\end{equation}
where $\alpha=2p-1$, see \cite{MR4897955}. When $\alpha\in[0,1]$, this ERW is also a usual SRRW with reinforcement parameter $\alpha$ and step distribution uniform on $\{-1,1\}$; when $\alpha\in(-1,0)$, it has the same law as the generalized SRRW on $\ZZ$ with reinforcement parameter $|\alpha|$, step distribution uniform on $\{-1,1\}$, and transformations $T_n=-\operatorname{Id}$.

For $m\in\ZZ_L$ and $k\in[L-1]$, write $\chi_k^{(L)}(m):=e^{2\mathrm{i}km\pi/L}$.
\begin{lemma}
\label{PSn012Rn0}
Let $S$ be a usual SRRW on $\ZZ_L$ with reinforcement parameter $\alpha\in[0,1)$ and step distribution $\mu$, and let $(R_n)_{n\geq 1}$ have the same parameter $\alpha$.
\begin{enumerate}[(i)]
\item If $L=2$ and $\mu(0)=\mu(1)=1/2$, then
\[
\EE\chi_1^{(2)}(S_n)=2\PP(S_n=0)-1=\mathrm{i}^nR_n(0), \qquad n\geq 1.
\]
\item If $L\geq 3$ and $\mu(-1)=\mu(1)=1/2$, then
\[
\EE\chi_k^{(L)}(S_n)=R_n(\cos(2k\pi/L)), \qquad k\in[L-1],\ n\geq 1.
\]
\end{enumerate}
\end{lemma}
\begin{proof}
Let $N_n(g)$ count the steps equal to $g$ by time $n$. In (i), the process $N_n(1)-N_n(0)$ is a usual SRRW on $\ZZ$ with step distribution uniform on $\{-1,1\}$. Since $S_n\equiv N_n(1)\pmod 2$ and $N_n(1)+N_n(0)=n$, (\ref{charERWpoly}) gives
\[
\EE\chi_1^{(2)}(S_n)
=e^{\mathrm{i}n\pi/2}\EE e^{\mathrm{i}\pi(N_n(1)-N_n(0))/2}
=\mathrm{i}^nR_n(0).
\]
This expectation also equals $2\PP(S_n=0)-1$. In (ii), apply (\ref{charERWpoly}) to $N_n(1)-N_n(-1)$ and use $S_n\equiv N_n(1)-N_n(-1)\pmod L$.
\end{proof}

The forest representation gives an exponential bound for $R_n(x)$ when $x\in(-1,1)$. See \cite[Figure 1]{MR4897955} for an illustration.
\begin{corollary}
\label{expdecayEP}
If $\alpha\in(-1,1)$, then for every $x\in(-1,1)$ and $n\geq 1$,
\[
|R_n(x)|\leq |x|^{(1-|\alpha|)n/8}+2e^{-(1-|\alpha|)n/18}.
\]
\end{corollary}
\begin{proof}
Put $\rho:=|\alpha|$, and let $\sigma=1$ when $\alpha\geq0$ and $\sigma=-1$ when $\alpha<0$. Let $S$ be the generalized SRRW on $\ZZ$ with reinforcement parameter $\rho$, step distribution uniform on $\{-1,1\}$, and transformations $T_n:=\sigma\operatorname{Id}$. Then $S$ has the same law as the ERW with memory parameter $p=(1+\alpha)/2$, and (\ref{charERWpoly}) gives
\[
\EE e^{\mathrm i tS_n}=R_n(\cos t),\qquad t\in\RR.
\]
By Proposition \ref{consgSRRWRRT},
\begin{equation}
    \label{SnsumCjng}
    S_n=\sum_{j=1}^n m_jg_j,\qquad
    m_j:=\sum_{k\in\CalC_{j,n}}\sigma^{\operatorname{depth}(k)},
\end{equation}
where $\operatorname{depth}(k)$ is the distance from $k$ to the root of its cluster, and $(g_j)_{j\geq 1}$ are i.i.d. symmetric $\{-1,1\}$-valued random variables independent of $\F_n$. For $t=\arccos x$, conditioning on $\F_n$ gives
\begin{equation}
    \label{Rn2jpiL}
|R_n(x)|\leq\EE\prod_{j=1}^n|\cos(tm_j)| \leq |x|^{(1-\rho)n/8}
   +\PP\left(I(n)\leq\frac{(1-\rho)n}{8}\right).
\end{equation}
Indeed, an isolated cluster is a single vertex, so $m_j=1$ and it contributes a factor $|\cos t|=|x|$, while every other factor is at most one. The conclusion follows from (\ref{upbdprobisolinear}) applied with $\alpha$ replaced by $\rho$.
\end{proof}

For the walk on $\ZZ_2$ in Lemma \ref{PSn012Rn0}(i) we have $\sigma=1$, so $m_j=|\CalC_{j,n}|$ and (\ref{SnsumCjng}) holds modulo two with i.i.d. uniform $\{0,1\}$-valued labels $g_j$. Given $\F_n$, the probability that $S_n=0$ is one if every cluster has even size, and is $1/2$ otherwise. Consequently,
\begin{equation}
    \label{Rn0allevenclu}
\mathrm{i}^nR_n(0)=2\PP(S_n=0)-1
=\PP(\text{every cluster in }\F_n\text{ has even size}).
\end{equation}
This identity motivates the following estimates.

\begin{proposition}
  \label{Rnsecbase}
For $\alpha\in[0,1]$, the elephant polynomials admit the representation
\begin{equation}
    \label{Rnxbasis}
R_n(x)=\sum_{k=0}^{\lfloor n/2\rfloor}(-1)^k\lambda_{n,k}x^{n-2k}(1-x^2)^k, \qquad x\in\RR,
\end{equation}
where the coefficients are non-negative and, for $n\geq 1$ and $0\leq k\leq\lfloor n/2\rfloor$,
\begin{equation}
    \label{estlambdank}
e^{-(1-\alpha)n/(2+\alpha)}\binom{n}{2k}\alpha^k
\leq\lambda_{n,k}\leq\binom{n}{2k}\alpha^k.
\end{equation}
Here and below, $\alpha^0=1$, including when $\alpha=0$.
\end{proposition}
\begin{remark}
When $\alpha=1$, these are the Chebyshev polynomials of the first kind, and $\lambda_{n,k}=\binom{n}{2k}$.
\end{remark}
\begin{proof}
The polynomials $x^{n-2k}(1-x^2)^k$, $0\leq k\leq\lfloor n/2\rfloor$, are linearly independent, as their lowest non-zero powers are distinct. Set $\lambda_{1,0}=1$ and use the convention $\lambda_{n,k}=0$ outside $0\leq k\leq\lfloor n/2\rfloor$. Substituting (\ref{Rnxbasis}) into (\ref{defelepoly}) gives
\begin{equation}
  \label{lambdarec}
\lambda_{n+1,k}
=\left(1+\frac{2\alpha k}{n}\right)\lambda_{n,k}
+\alpha\left(1-\frac{2(k-1)}{n}\right)\lambda_{n,k-1}\geq 0,
\end{equation}
for $n\geq 1$ and $0\leq k\leq\lfloor(n+1)/2\rfloor$. This proves (\ref{Rnxbasis}) and the non-negativity of its coefficients.

Write $B_{n,k}:=\binom{n}{2k}$, with binomial coefficients outside their usual range equal to zero. Pascal's identity gives
\begin{equation}
  \label{lamnkupn1}
\left(1+\frac{2\alpha k}{n}\right)B_{n,k}
+\left(1-\frac{2(k-1)}{n}\right)B_{n,k-1}
=B_{n+1,k}-(1-\alpha)\binom{n-1}{2k-1}.
\end{equation}
Moreover,
\begin{equation}
  \label{arigeomeanine}
0\leq\frac{\binom{n-1}{2k-1}}{\binom{n+1}{2k}}
=\frac{2k(n+1-2k)}{n(n+1)}\leq\frac13.
\end{equation}
For $n\geq 3$, the last inequality follows from $2k(n+1-2k)\leq(n+1)^2/4$; the cases $n=1,2$ follow directly. Let $q_\alpha:=(2+\alpha)/3$. The ratio bound (\ref{arigeomeanine}) gives
\[
q_\alpha B_{n+1,k}
\leq B_{n+1,k}-(1-\alpha)\binom{n-1}{2k-1}
\leq B_{n+1,k}.
\]
 Using (\ref{lambdarec}) and (\ref{lamnkupn1}), by induction, we can obtain
\begin{equation}
  \label{lamnklown1}
q_\alpha^{n-1}\alpha^k B_{n,k}\leq\lambda_{n,k}\leq\alpha^k B_{n,k}.
\end{equation}
Indeed, for $\alpha=0$, (\ref{lambdarec}) gives $\lambda_{n,0}=1$ and $\lambda_{n,k}=0$ for $k\geq1$. For $\alpha>0$, the induction follows by multiplying the two terms in (\ref{lamnkupn1}) by $\alpha^k$ and applying the preceding bounds.  Finally, $-\log(1-t)\leq t/(1-t)$ for $0\leq t<1$. Taking $t=(1-\alpha)/3$ gives
\begin{equation}
  \label{lnx1xinalpha}
-\log q_\alpha
=-\log\left(1-\frac{1-\alpha}{3}\right)
\leq\frac{1-\alpha}{2+\alpha},
\end{equation}
which proves (\ref{estlambdank}).
\end{proof}

\begin{proof}[Proof of Corollary \ref{twopointcor}]
If $n$ is odd, then $\F_n$ has a cluster of odd size, so (\ref{Rn0allevenclu}) gives $\PP(S_n=0)=1/2$. For every $n\geq 1$, Proposition \ref{Rnsecbase} gives
\[
2\PP(S_{2n}=0)-1=(-1)^nR_{2n}(0)=\lambda_{2n,n}
\in\left[e^{-2(1-\alpha)n/(2+\alpha)}\alpha^n,\alpha^n\right].
\]
Since $2/(2+\alpha)\leq1$, this implies (\ref{Z2estP2n0}). For fixed $n$, its logarithm shows that
\[
\log(2\PP(S_{2n}=0)-1)=n\log\alpha+O(n),\qquad\alpha\to0+,
\]
which proves the first limit. For the second limit, use $\alpha^n\leq\alpha$ for the lower bound and $1-y^n\leq n(1-y)$ for $y\in[0,1]$ for the upper bound. Equation (\ref{Z2estP2n0}) then gives
\[
\frac{1-\alpha}{2}\leq1-\PP(S_{2n}=0)
\leq\frac n2\left(1-\alpha e^{-(1-\alpha)}\right).
\]
Since $e^{-(1-\alpha)}\geq\alpha$, the last term is at most
\[
\frac n2(1-\alpha^2)
=\frac n2(1-\alpha)(1+\alpha)\leq n(1-\alpha).
\]
Taking logarithms proves the second limit.
\end{proof}

\section{Acknowledgments}

Yuval Peres is supported by the National Natural Science Foundation of China under Grant Number W2531011. Shuo Qin is supported by the China Postdoctoral Science Foundation under Grant Number 2025M773086, 2026T190815, and National Natural Science Foundation of China under Grant Number 12271284.

\bibliographystyle{plain}
\bibliography{math_ref}

\end{document}